\documentclass[12pt]{article}
\usepackage{amsmath}
\usepackage{amssymb}
\usepackage{latexsym,bm}
\usepackage{amsthm}
\usepackage{graphicx}

\usepackage[top=2.4cm,bottom=2cm,left=2.4cm,right=2.4cm]{geometry}
\usepackage{algorithm}
\usepackage{algorithmic}
\usepackage{cite}
\allowdisplaybreaks
\usepackage{multirow}

\providecommand{\U}[1]{\protect\rule{.1in}{.1in}}
\providecommand{\U}[1]{\protect\rule{.1in}{.1in}}
\providecommand{\U}[1]{\protect\rule{.1in}{.1in}}
\providecommand{\U}[1]{\protect\rule{.1in}{.1in}}

\newtheorem{theorem}{Theorem}[section]

\theoremstyle{definition}

\theoremstyle{remark}
\newtheorem{remark}{Remark}[section]
\numberwithin{equation}{section}
\theoremstyle{example}

\numberwithin{equation}{section}

\begin{document}

\title{An inertial alternating direction method of multipliers for solving a two-block separable convex minimization problem}

\author{ Yang Yang$^{a}$, Yuchao Tang$^{a}$\footnote{Corresponding author: Yuchao Tang. Email address: hhaaoo1331@163.com} \\
\small ${^a}$Department of Mathematics, Nanchang University,\\
\small  Nanchang 330031, P.R. China
}

 \date{}

\maketitle

{}\textbf{Abstract.}
The alternating direction method of multipliers (ADMM) is a widely used method for solving many convex minimization models arising in signal and image processing. In this paper, we propose an inertial ADMM for solving a two-block separable convex minimization problem with linear equality constraints. This algorithm is obtained by making use of the inertial Douglas-Rachford splitting algorithm to the corresponding dual of the primal problem. We study the convergence analysis of the proposed algorithm in infinite-dimensional Hilbert spaces. Furthermore, we apply the proposed algorithm on the robust principal component pursuit problem and also compare it with other state-of-the-art algorithms. Numerical results demonstrate the advantage of the proposed algorithm.



\section{Introduction}

In this paper, we consider the convex optimization problem of two-block separable objective functions with linear equality constraints:
\begin{equation}\label{problem1}
\begin{aligned}
 \min_{u\in H_1, v\in H_2}\, & F(u) + G(v) \\
\textrm{s.t.}\, & Mu+Nv=b,
\end{aligned}
\end{equation}
where $b\in H$, $M: H_{1}\rightarrow H$ and $N: H_{2}\rightarrow H$ are bounded linear operators, $F:H_{1}\rightarrow (-\infty, +\infty]$ and $G:H_{2}\rightarrow (-\infty, +\infty]$ are proper, convex and lower semi-continuous functions (not necessarily smooth), $H$, $H_{1}$ and $H_{2}$ are real Hilbert spaces. Specifically, $H=R^{p}$, $H_1 = R^{n}$, $H_2 = R^{m}$. Then $M\in R^{p\times n}$ and $N\in R^{p\times m}$. We use the notations of Hilbert spaces for generality. Many problems in signal and image processing, medical image reconstruction, machine learning, and many other areas are special case of (\ref{problem1}). When $N=-I$ and $b=0$, then (\ref{problem1}) degenerates into an important special case of it as follows,
\begin{equation}\label{problem2}
\begin{aligned}
 \min_{u\in H_1, v\in H_2}\, & F(u) + G(v) \\
\textrm{s.t.}\, & Mu=v,
\end{aligned}
\end{equation}
which is equivalent to
\begin{equation}\label{problem22}
\min_{u\in H_{1}} F(u)+G(Mu).
\end{equation}

The alternating direction method of multipliers (ADMM) is a popular way to solve (\ref{problem1}) and (\ref{problem2}). It has been attracted much attention because of its simplicity and efficiency. The ADMM can be dated back to the work of  Glowinski and Marroco \cite{Glowinski1975}, and Gabay and Mercier \cite{Gabay1976}. The classical formulation of the ADMM for solving (\ref{problem1}) can be presented below.
\begin{equation}\label{ADMM}
   \left\{
\begin{aligned}
u^{k+1} &  = \arg\min_{u} \{ F(u) + \langle y^k , Mu \rangle + \frac{\gamma}{2} \| Mu + Nv^k -b \|^2  \}, \\
v^{k+1} & = \arg\min_{v} \{ G(v) + \langle y^k , Nv \rangle + \frac{\gamma}{2} \| Mu^{k+1} + Nv -b \|^2  \}, \\
y^{k+1} & = y^k + \gamma (Mu^{k+1}+ N v^{k+1}-b),
\end{aligned}
\right.
\end{equation}
where $\gamma >0$. It is well-known that the ADMM can be interpreted as an application of the Douglas-Rachford splitting algorithm  to the dual of problem  (\ref{problem1}). See, for instance \cite{Eckstein1992}. It is worth mentioning that the famous split Bregman method \cite{goldstein2009} is also equivalent to the ADMM. See, e.g., \cite{esser2009}. A comprehensive review of the ADMM with applications in various convex optimization problems can be found in \cite{boyd1}. For the convergence and convergence rate analysis of ADMM for solving (\ref{problem1}) and (\ref{problem22}), we refer interested readers to \cite{He2012SIAMNA, Monteiro2013SIAMJO,Fang2015MPC,He2015NM} for more details. Also, many efforts have been tried to extend the ADMM for solving multi-block separable convex minimization problem, see for instance \cite{Hebs2015SIAMJO,Hebs2015IMAJNA,Chench2016MP}. In this paper, we focus on the general case of the two-block separable convex minimization problem (\ref{problem1}).

The generalized ADMM (GADMM) proposed by Eckstein and Bertsekas \cite{Eckstein1992} is an efficient and simple acceleration scheme of the classical ADMM (\ref{ADMM}) for solving (\ref{problem22}). It is easy to extend the GADMM to solve (\ref{problem1}), and the iterative scheme of the GADMM reads as
\begin{equation}\label{GADMM}
   \left\{
\begin{aligned}
u^{k+1} &  = \arg\min_{u} \{ F(u) + \langle y^k , Mu \rangle + \frac{\gamma}{2} \| Mu + Nv^k -b \|^2  \}, \\
v^{k+1} & = \arg\min_{v} \{ G(v) + \langle y^k , Nv \rangle + \frac{\gamma}{2} \| N(v-v^{k}) + \lambda_{k}(Mu^{k+1} + Nv^{k} -b) \|^2  \}, \\
y^{k+1} & = y^k + \gamma ( N(v^{k+1}-v^{k}) + \lambda_{k}(Mu^{k+1} + Nv^{k} -b)),
\end{aligned}
\right.
\end{equation}
where $\gamma>0$ and $\lambda_{k}\in(0,2)$. Let $\lambda_{k} = 1$, then the GADMM (\ref{GADMM}) reduces to the classical ADMM (\ref{ADMM}). There are many works demonstrated that the GADMM (\ref{GADMM}) can numerically accelerate the classical ADMM (\ref{ADMM}) with $\lambda_{k}$ belongs to $(1,2)$. See, for example \cite{Caixj2013,Fang2015MPC}.

In recent years, the inertial method becomes more and more popular. It can be used to accelerate the first-order method for solving nonsmooth convex optimization problems. It is closely related to the famous Nesterov's accelerate method, which utilizes the current iteration information and the previous iteration information to update the new iteration. Many inertial algorithms have been proposed and studied, such as inertial proximal point algorithm \cite{Alvarez2001,attouch:hal-01708905}, inertial forward-backward splitting algorithm \cite{Attouch2018SJO,Attouch2019AMO}, inertial forward-backward-forward splitting algorithm \cite{Bot2016NA}, inertial three-operator splitting algorithm \cite{Cui2019}, etc. There are also several attempts to introduce the inertial method to the ADMM. In particular, Chen et al. \cite{Chen2015} proposed an inertial proximal ADMM by combining the proximal ADMM \cite{Attouch2008PJO} and the inertial proximal point algorithm \cite{Alvarez2001}. The detail of the inertial proximal ADMM is presented below.
\begin{equation}\label{inertial-ADMM-chen}
   \left\{
\begin{aligned}
& (\bar{u}^{k},\bar{v}^{k},\bar{y}^{k})=(u^{k},v^{k},y^{k}) + \alpha_{k}(u^{k}-u^{k-1},v^{k}-v^{k-1},y^{k}-y^{k-1}), \\
& u^{k+1}=\arg\min_{u}\{F(u)+\langle \bar{y}^{k},Mu\rangle+\frac{\gamma}{2}\|Mu+N\bar{v}^{k}-b\|^{2}+\frac{1}{2}\|u-\bar{u}^{k}\|^{2}_{S}\}, \\
& y^{k+1}=\bar{y}^{k}+\gamma(Mu^{k+1}+N\bar{v}^{k}-b), \\
& v^{k+1}=\arg\min_{v}\{G(v)+\langle y^{k+1},Nv\rangle+\frac{\gamma}{2}\|Mu^{k+1}+Nv-b\|^{2}+\frac{1}{2}\|v-\bar{v}^{k}\|^{2}_{T}\},
\end{aligned}
\right.
\end{equation}
where $\alpha_{k}$ is a sequence of non-negative parameters, which usually called inertial parameters, $S$ and $T$ are symmetric and positive semidefinite matrices, and $\gamma>0$ is a penalty parameter. Let the matrices $S=T=0$ and $\alpha_k =0$, then the inertial proximal ADMM (\ref{inertial-ADMM-chen}) recovers the following ADMM type algorithm, which is studied in \cite{Caixj2013}.
\begin{equation}\label{ADMM-cai}
   \left\{
\begin{aligned}
& u^{k+1}=\arg\min_{u}\{F(u)+\langle y^{k},Mu\rangle+\frac{\gamma}{2}\|Mu+Nv^{k}-b\|^{2}\}, \\
& y^{k+1}=y^{k}+\gamma(Mu^{k+1}+Nv^{k}-b), \\
& v^{k+1}=\arg\min_{v}\{G(v)+\langle y^{k+1},Nv\rangle+\frac{\gamma}{2}\|Mu^{k+1}+Nv-b\|^{2}\}.
\end{aligned}
\right.
\end{equation}
The difference between (\ref{ADMM}) and (\ref{ADMM-cai}) is that the update order of the former is $u^{k+1} \rightarrow v^{k+1} \rightarrow y^{k+1}$, but the update order of the later is $u^{k+1} \rightarrow y^{k+1} \rightarrow v^{k+1}$. In contrast, Bo\c{t} and Csetnek \cite{Bot2016MTA} proposed an inertial ADMM for solving the convex optimization problem (\ref{problem22}), which was based on the inertial Douglas-Rachford splitting algorithm \cite{Bot2015AMC}. It takes the form of
\begin{equation}\label{inertial-ADMM-bot}
   \left\{
\begin{aligned}
& u^{k+1} = \arg\min_{u} \{ F(u) + \langle y^k - \alpha_k (y^k - y^{k-1}) - \gamma \alpha_k (v^k - v^{k-1}), M u \rangle \\
& \quad \quad \quad + \frac{\gamma}{2} \| Mu - v^k\|^2 \}, \\
& \overline{v}^{k+1} = \alpha_{k+1}\lambda_k (Mu^{k+1} - v^k ) + \frac{ (1-\lambda_k)\alpha_k \alpha_{k+1} }{\gamma} (y^k - y^{k-1} + \gamma (v^k - v^{k-1})), \\
& v^{k+1} = \arg\min_{v} \{  G(v + \overline{v}^{k+1}) + \langle  - y^k - (1-\lambda_k)\alpha_k ( y^k - y^{k-1} +\gamma (v^k - v^{k-1}) ), v \rangle \\
& \quad \quad \quad + \frac{\gamma}{2} \| v  - \lambda_k Mu^{k+1} - (1-\lambda_k)v^k \|^2 \}, \\
& y^{k+1} = y^k + \gamma ( \lambda_k Mu^{k+1} + (1-\lambda_k)v^k - v^{k+1}  ) \\
& \quad \quad \quad + (1-\lambda_k)\alpha_k ( y^k - y^{k-1} + \gamma (v^k - v^{k-1}) ).
\end{aligned}
\right.
\end{equation}
Let $\alpha_k =0$ and $\lambda_k =1$, then (\ref{inertial-ADMM-bot}) becomes the classical ADMM for solving the convex optimization problem (\ref{problem22}).  Bo\c{t} and Csetnek analyzed the convergence of the sequences generated by the inertial ADMM  (\ref{inertial-ADMM-bot}) under mild conditions on the parameters $\alpha_k$ and $\lambda_k$.

The purpose of this paper is to introduce an inertial ADMM for solving the general two-block separable convex optimization problem (\ref{problem1}). We prove the convergence of the sequences generated by the proposed inertial ADMM. Finally, we conduct numerical experiments on robust principal component pursuit (RPCP) problem and compare the proposed algorithm with the classical ADMM (\ref{ADMM}), the GADMM (\ref{GADMM}) and the inertial proximal ADMM (\ref{inertial-ADMM-chen}) to demonstrate the advantage of the proposed algorithm.

We would like to highlight the contributions of this paper. (i) An inertial ADMM is developed to solve the convex minimization problem (\ref{problem1}); (ii) The convergence of the proposed inertial ADMM is studied in infinite-dimensional Hilbert space; (iii) The effectiveness and efficiency of the proposed inertial ADMM is demonstrated by applying to the RPCP problem.

The structure of this paper is as follows. In the next section, we summarize some notations and definitions that will be used in the following sections. We also recall the inertial Douglas-Rachford splitting algorithm. In Section 3, we introduce an inertial ADMM and study its convergence results in detail. In Section 4, some numerical experiments for solving the RPCP problem (\ref{e4.2}) are conducted to demonstrate the efficiency of the proposed algorithm. Finally, we give some conclusions.

\section{Preliminaries\label{sec:pre}}

In this section, we recall some concepts of monotone operator theory and convex analysis in Hilbert spaces. Most of them can be found in \cite{bauschkebook2017}.
Let $H_{1}$ and $H_{2}$ be real Hilbert spaces with $inner\ product \langle\cdot,\cdot\rangle$ and associated $norm \|\cdot\|=\sqrt{\langle\cdot,\cdot\rangle}$. $x_{k}\rightharpoonup x$ stands for $\{x_{k}\}$ converging weakly to $x$, and $x_{k}\rightarrow x$ stands for $\{x_{k}\}$ converging strongly to $x$. Let $A:H_{1}\rightarrow H_{2}$ be a linear continuous operator and its  adjoint operator be $A^{\ast}:H_{2}\rightarrow H_{1}$ is the unique operator that satisfies $\langle A^{\ast}y,x\rangle=\langle y,Ax\rangle$ for all $x\in H_{1}$ and $y\in H_{2}$.

Let $A:H\rightarrow2^{H}$ be a set-valued operator. We denote its set of zeros by $zer A=\{x\in H:0\in Ax\}$, by $Gra A=\{(x,u)\in H\times H:u\in Ax\}$ denote its graph and its inverse operator denote by $A^{-1}:H\rightarrow2^{H}$, where $(u,x)\in Gra A^{-1}$ if and only if $(x,u)\in Gra A$. We say that $A$ is monotone if $\langle x-y,u-v\rangle\geq 0$, for all $(x,u),(y,v)\in Gra A$. $A$ is said to be maximally monotone if its graph is not contained in the graph of any other monotone operator. Letting $\gamma>0$, the resolvent of $\gamma A$ is defined by
\begin{equation}
J_{\gamma A}=(I+\gamma A)^{-1}.
\end{equation}
Moreover, if $A$ is maximally monotone, then $J_{\gamma A}:H\rightarrow H$ is single-valued and maximally monotone.

The operator $A$ is uniformly monotone with modulus $\phi_{A}:R_{+}\rightarrow[0,+\infty)$, if $\phi_{A}$ is increasing, $\phi_{A}$ vanishes only at $0$, and
\begin{equation}
\langle x-y,u-v\rangle\geq \phi_{A}(\|x-y\|).
\end{equation}
for all $(x,u),(y,v)\in Gra A$. Moreover, If $\phi_{A}=\gamma\|x-y\|^{2}(\gamma>0)$, then $A$ is $\gamma$-strongly monotone.

For a function $f:H\rightarrow \bar{R}$, where $\bar{R}:=R\bigcup \{+\infty\}$, We define $dom f=\{x\in H:f(x)<+\infty\}$ as its effective domain and say that $f$ is proper if $dom f\neq \emptyset$ and $f(x)\neq -\infty$ for all $x\in H$. Let $f:H\rightarrow\bar{R}$. The conjugate of $f$ is $f^{\ast}:H\rightarrow\bar{R}:u\mapsto \sup_{x\in H} \{\langle x,u\rangle -f(x)\}$. We denote by $\Gamma_{0}(H)$ the family of proper, convex and lower semi-continuous extended real-valued functions defined on $H$. Let $f\in \Gamma_{0}(H)$. The subdifferential of $f$ is the set-valued operator $\partial f(x):=\{v\in H:f(y)\geq f(x)+\langle v,y-x\rangle \, \forall y\in H \}$. Moreover, $\partial f$ is a maximally monotone operator and it holds $(\partial f)^{-1}=\partial f^{\ast}$. For every $x\in H$ and arbitrary $\gamma>0$, we have
\begin{equation}
Prox_{\gamma f}(x)=\arg\min_{y}\{\gamma f(y)+\frac{1}{2}\|y-x\|^{2} \}.
\end{equation}
Here $prox_{\gamma f}(x)$ is the proximal point of parameter $\gamma$ of $f$ at $x$. And the proximal point operator of $f$ satisfies $Prox_{\gamma f}(x)=(I+\gamma\partial f)^{-1}=J_{\gamma\partial f}$.

Let $f\in \Gamma_{0}(H)$. Then $f$ is uniformly convex with modulus $\phi:R_{+}\rightarrow[0,+\infty)$ if $\phi$ is increasing, $\phi$ vanishes only at $0$, and
\begin{equation}\label{def7}
f(\alpha x+(1-\alpha)y)+\alpha(1-\alpha)\phi(\|x-y\|)\leq\alpha f(x)+(1-\alpha)f(y).
\end{equation}
for all $x,y\in domf,\forall \alpha\in(0,1)$. If (\ref{def7}) holds with $\phi=\frac{\beta}{2}(\cdot)^{2}$ for some $\beta\in(0,+\infty)$, then $f$ is $\beta$-strongly convex. This also means that $\partial f$ is $\beta$-strongly monotone (if $f$ is uniformly convex, then $\partial f$ is uniformly monotone).

At the end of this section, we recall the main results of the inertial Douglas-Rachford splitting algorithm in \cite{Bot2015AMC}.

\begin{theorem}(\cite{Bot2015AMC})\label{theo1}
Let $A,B:H\rightarrow 2^{H}$ be maximally monotone operators. Assume $zer(A+B)\neq\emptyset$. For any given $w^{0},w^{1}\in H$, define the iterative sequences as follows:
\begin{equation}\label{algo1}
   \left\{
\begin{aligned}
& y^{k}=J_{\gamma B}(w^{k}+\alpha_{k}(w^{k}-w^{k-1})), \\
& x^{k}=J_{\gamma A}(2y^{k}-w^{k}-\alpha_{k}(w^{k}-w^{k-1}), \\
& w^{k+1}=w^{k}+\alpha_{k}(w^{k}-w^{k-1})+\lambda_{k}(x^{k}-y^{k}),
\end{aligned}
\right.
\end{equation}
where $\gamma > 0$, $\{\alpha_{k}\}_{k\geq1}$ is nondecreasing with $\alpha_{1}=0$ and $0\leq\alpha_{k}\leq\alpha<1$ for every $k\geq 1$ and $\lambda,\delta,\sigma>0$ are such that
\begin{equation}\label{eq2.5}
\delta>\frac{\alpha^{2}(1+\alpha)+\alpha\sigma}{1-\alpha^{2}}, \quad \textrm{ and } \quad 0<\lambda\leq\lambda_{k}\leq2\frac{\delta-\alpha[\alpha(1+\alpha)+\alpha\delta+\sigma]}{\delta[1+\alpha(1+\alpha)+\alpha\delta+\sigma]},\quad \forall k\geq 1.
\end{equation}
Then there exists $x\in H$ such that the following statements are true:
\vskip 1mm
\rm (1) $J_{\gamma B}x\in zer(A+B)$;

\rm(2) $\Sigma_{k\in N}\|w^{k+1}-w^{k}\|^{2}<+\infty$;

\rm(3) $\{w^{k}\}_{k\geq0}$ converges weakly to $x$;

\rm(4) $y^{k}-x^{k}\rightarrow0$ as $k\rightarrow+\infty$;

\rm(5) $\{y^{k}\}_{k\geq1}$ converges weakly to $J_{\gamma B}x$;

\rm(6) $\{x^{k}\}_{k\geq1}$ converges weakly to $J_{\gamma B}x=J_{\gamma A}(2J_{\gamma B}x-x)$;

\rm(7) if $A$ or $B$ is uniformly monotone, then $\{y^{k}\}_{k\geq1}$ and $\{x^{k}\}_{k\geq1}$ converges strongly to unique point in $zer(A+B)$.
\end{theorem}

By making full use of the inertial Krasnoselskii-Mann iteration scheme of Maing\'{e} \cite{Mainge2008JCAM}, we present another convergence theorem of the inertial Douglas-Rachford splitting algorithm (\ref{algo1}). Since the proof is similar to Theorem \ref{theo1}, so we omit it here.

\begin{theorem}\label{theo2.2}
Let $A,B:H\rightarrow 2^{H}$ be maximally monotone operators. Assume $zer(A+B)\neq\emptyset$. Let the sequences $\{y^{k}\}$, $\{x^{k}\}$ and $\{w^{k}\}$ be defined by (\ref{algo1}). Let the sequences $\{\lambda_{k}\}$ and $\{\alpha_{k}\}$ satisfy the following conditions:

\rm(i) $0\leq\alpha_{k}\leq\alpha<1$, $0<\underline{\lambda}\leq\lambda_{k}\leq\overline{\lambda}<2$;

\rm(ii) $\sum_{k=0}^{+\infty}\alpha_{k}\|w^{k}-w^{k-1}\|^{2}<+\infty$.

Then there exists $x\in H$ such that the following statements are true:
\vskip 1mm
\rm (1) $J_{\gamma B}x\in zer(A+B)$;

\rm(2) $\{w^{k}\}_{k\geq0}$ converges weakly to $x$;

\rm(3) $y^{k}-x^{k}\rightarrow0$ as $k\rightarrow+\infty$;

\rm(4) $\{y^{k}\}_{k\geq1}$ converges weakly to $J_{\gamma B}x$;

\rm(5) $\{x^{k}\}_{k\geq1}$ converges weakly to $J_{\gamma B}x=J_{\gamma A}(2J_{\gamma B}x-x)$;

\rm(6) if $A$ or $B$ is uniformly monotone, then $\{y^{k}\}_{k\geq1}$ and $\{x^{k}\}_{k\geq1}$ converges strongly to unique point in $zer(A+B)$.
\end{theorem}

\begin{remark}\label{rema1}
According to \cite{Bot2015AMC}, the condition $\alpha_{1}=0$ in the above Theorem \ref{theo1} can be replaced by the assumption $w^{0}=w^{1}$.
\end{remark}
\vskip 2mm

\section{Inertial ADMM for solving the two-block separable convex minimization problem }

In this section, we present the main results of this paper.
Let $H$, $H_{1}$ and $H_{2}$ be real Hilbert spaces, $F\in \Gamma_{0}(H_{1})$, $G\in \Gamma_{0}(H_{2})$, $b\in H$, $M:H_{1}\rightarrow H$ and $N:H_{2}\rightarrow H$ are linear continuous operators.
The dual problem of (\ref{problem1}) is
\begin{equation}\label{1.8}
\max_{y\in H}-F^{\ast}(-M^{\ast}y)-G^{\ast}(-N^{\ast}y)-\langle y,b\rangle,
\end{equation}
where $F^{\ast}$ and $G^{\ast}$ are the Fenchel-conjugate functions of $F$ and $G$, respectively. We consider solving the convex optimization problems (\ref{problem1}) and its  dual problem (\ref{1.8}). Let $v(P)$ and $v(D)$ be the optimal objective values of the above two problems respectively, the situation $v(P)\geq v(D)$, called in the literature weak duality, always holds. We introduce the $Attouch-Br\acute{e}zies$ condition, that is
\begin{equation}\label{3.1}
0 \in sqri(M(dom\, F)-N(dom\, G)).
\end{equation}
For arbitrary convex set $C\subseteq H$, we define its $strong\ quasi-relative\ interior$ as
\begin{equation}\label{3.2}
sqri\ C:=\{x\in C:\cup_{\lambda >0}\lambda(C-x)\ is\ a\ closed\ linear\ subspace\ of\ H\}.
\end{equation}
If (\ref{3.1}) holds, then we have strong duality, which means that $v(P) = v(D)$ and (\ref{1.8}) has an optimal solution.

Next, we introduce the main algorithm in this paper.
\vskip 2mm
\renewcommand{\algorithmicrequire}{\textbf{Input:}}
\renewcommand{\algorithmicensure}{\textbf{Output:}}
\begin{algorithm}[htb]
\caption{An inertial alternating direction method of multipliers (iADMM)}
\label{algo2}
\begin{algorithmic}[1]
\REQUIRE
\vskip 2mm
For arbitrary $y^{1}\in H$, $v^{1}\in H_{2}$, $p^{1}=0$, choose $\gamma$, $\alpha_{k}$ and $\lambda_{k}$.\\
For $k= 1,2,3, \cdots$, compute
\quad  \STATE $u^{k+1}=\arg\min_{u}\{ F(u)+\langle y^{k},Mu\rangle+\frac{\gamma}{2}\|Mu+Nv^{k}-b\|^{2}\}$;

\STATE $v^{k+1} =\arg\min_{v}\{G(v)+\langle y^{k}+\alpha_{k+1}p^{k},Nv\rangle +\frac{\gamma}{2}\|N(v-v^{k})+(1+\alpha_{k+1})\lambda_{k}(Mu^{k+1}+Nv^{k}-b)\|^{2}\}$;

\STATE $y^{k+1}=y^{k}+\alpha_{k+1}p^{k}+\gamma[N(v^{k+1}-v^{k})+(1+\alpha_{k+1})\lambda_{k}(Mu^{k+1}+Nv^{k}-b)]$;\\

\STATE $p^{k+1}=\alpha_{k+1}[p^{k}+\gamma\lambda_{k}(Mu^{k+1}+Nv^{k}-b)]$;

Stop when a given stopping criterion is met.
\ENSURE $u^{k+1},v^{k+1}$ and $y^{k+1}$.
\end{algorithmic}
\end{algorithm}
\vskip 2mm

In order to analyze the convergence of Algorithm \ref{algo2}, we define the Lagrangian function of problem (\ref{problem1}) as follows:
\begin{equation}\label{1.3}
L(u,v,y)=F(u)+G(v)+\langle y,Mu+Nv-b\rangle,
\end{equation}
where $y$ is a Lagrange multiplier.  Assuming $(u^{\ast},v^{\ast})$ is a optimal solution of the optimization problem (\ref{problem1}), there exits a vector $y^{*}$, according to KKT condition, we have
\begin{equation}\label{1.4}
\begin{aligned}
& 0\in \partial F(u^{\ast})+M^{\ast}y^{\ast}, \\
& 0\in \partial G(v^{\ast})+N^{\ast}y^{\ast}, \\
& Mu^{\ast}+Nv^{\ast}-b=0.
\end{aligned}
\end{equation}
Moreover, point pairs $(u^{\ast},v^{\ast},y^{\ast})$ are saddle points of Lagrange function, that is
\begin{equation}\label{1.7}
L(u^{\ast},v^{\ast},y)\leq L(u^{\ast},v^{\ast},y^{\ast})\leq L(u,v,y^{\ast})\quad \forall(u,v,y)\in H_{1}\times H_{2}\times H.
\end{equation}

In order to analyze the convergence of the proposed Algorithm \ref{algo2} in Hilbert spaces, we show that the iterative sequences generated by Algorithm \ref{algo2} are instances of the inertial Douglas-Rachford splitting algorithm (\ref{algo1}) applied to the dual problem (\ref{1.8}). In detail, we show that Algorithm \ref{algo2} could be derived from the inertial Douglas-Rachford splitting  algorithm (\ref{algo1}). Then we use Theorem \ref{theo1} to obtain the convergence of  the proposed Algorithm \ref{algo2}. Now, we are ready to present the main convergence theorem of Algorithm \ref{algo2}.

\begin{theorem}\label{theo2}
Assuming (\ref{problem1}) has an optimal solution and the condition (\ref{3.1}) is satisfied. Let the bounded linear operators $M$ and $N$ satisfy the condition that $\exists \theta_1 >0, \theta_2 >0$ such that $\|Mx\| \geq \theta_1 \|x\|$ and $\|Nx\| \geq \theta_2 \|x\|$, for all $x\in H$. Consider the sequence generated by Algorithm \ref{algo2}. Let $\gamma>0$, $\{\alpha_{k}\}_{k\geq1}$ nondecreasing with $0\leq\alpha_{k}\leq\alpha<1$, $\{\lambda_{k}\}_{k\geq 1}$ and $\lambda,\sigma,\delta>0$ such that
$$
\delta>\frac{\alpha^{2}(1+\alpha)+\alpha\sigma}{1-\alpha^{2}}\quad and \quad 0<\lambda\leq\lambda_{k}\leq2\frac{\delta-\alpha[\alpha(1+\alpha)+\alpha\delta+\sigma]}{\delta[1+\alpha(1+\alpha)+\alpha\delta+\sigma]}\quad \forall k\geq1,
$$
Then there exists a point pair $(u^{\ast},v^{\ast},y^{\ast})$, which is the saddle point of Lagrange function, where $(u^{\ast},v^{\ast})$ is the optimal solution of (\ref{problem1}), $y^{\ast}$ is the optimal solution of (\ref{1.8}), and $v(P)=v(D)$. The following statements are true:
\vskip 1mm
\rm (i)\, $(u^{k},v^{k})_{k\geq1}$ converges weakly to $(u^{\ast},v^{\ast})$;

%

\rm (ii)\, $(Mu^{k+1}+Nv^{k})_{k\geq1}$ converges strongly to $b$;

\rm (iii)\, $(y^{k})_{k\geq1}$ converges weakly to $y^{\ast}$;

\rm (iv)\, if $F^{\ast}$ or $G^{\ast}$ is uniformly convex, the $(y^{k})_{k\geq1}$ converges strongly to unique optimal solution of (D);

\rm (v)\, $\lim_{k\rightarrow+\infty}F(u^{k+1})+G(v^{k})=v(P)=v(D)= \lim_{k\rightarrow+\infty}(-F^{\ast}(-M^{\ast}x^{k})-G^{\ast}(-N^{\ast}y^{k})-\langle y^{k},b\rangle)$,where the sequence $(x^{k})_{k\geq1}$ is defined by
$$
x^{k}=y^{k}+\gamma Mu^{k+1}+\gamma Nv^{k}-\gamma b,
$$
and $(x^{k})_{k\geq1}$ converges weakly to $y^{\ast}$.
\end{theorem}
\vskip 2mm

\begin{proof}
Let
\begin{equation}\label{e3.3}
A=\partial(F^{\ast}\circ(-M^{\ast})),\ \textrm{ and } \ B=\partial(G^{\ast}\circ(-N^{\ast}))+b.
\end{equation}
From the first step of iteration scheme (\ref{algo1}), we have
\begin{align}\label{eq1}
y^{k}&=J_{\gamma B}(w^{k}+\alpha_{k}(w^{k}-w^{k-1}))\nonumber\\
&=prox_{\gamma(G^{\ast}(-N^{\ast})+\langle\cdot,b\rangle)}(w^{k}+\alpha_{k}(w^{k}-w^{k-1}))\nonumber\\
&=\arg\min_{y}\{\gamma(G^{\ast}(-N^{\ast}y)+\langle y,b\rangle)+\frac{1}{2}\|y-w^{k}-\alpha_{k}(w^{k}-w^{k-1})\|^{2}\}.
\end{align}
By the first-order optimality condition, we obtain from (\ref{eq1}) that
\begin{equation}\label{e3.5}
0\in -\gamma N\partial G^{\ast}(-N^{\ast}y^{k})+\gamma b+y^{k}-w^{k}-\alpha_{k}(w^{k}-w^{k-1}).
\end{equation}
We introduce the sequence $\{v^{k}\}_{k\geq 1}$ by
\begin{equation}\label{e3.6}
v^{k}\in \partial G^{\ast}(-N^{\ast}y^{k}),
\end{equation}
then, we have
\begin{equation}\label{E3.1}
0=-\gamma Nv^{k}+\gamma b+y^{k}-w^{k}-\alpha_{k}(w^{k}-w^{k-1}) \quad and \quad -N^{\ast}y^{k}\in \partial G(v^{k}).
\end{equation}
From the second step of iteration scheme (\ref{algo1}), we have
\begin{align}
x^{k}&=J_{\gamma A}(2y^{k}-w^{k}-\alpha_{k}(w^{k}-w^{k-1}))\nonumber\\
&=prox_{\gamma F^{\ast}(-M^{\ast})}(2y^{k}-w^{k}-\alpha_{k}(w^{k}-w^{k-1}))\nonumber\\
&=\arg\min_{x}\{\gamma F^{\ast}(-M^{\ast}x)+\frac{1}{2}\|x-2y^{k}+w^{k}+\alpha_{k}(w^{k}-w^{k-1})\|^{2}\}.
\end{align}
According to the first-order optimality condition, there are also
\begin{equation}\label{e3.9}
0\in -\gamma M\partial F^{\ast}(-M^{\ast}x^{k})+x^{k}-2y^{k}+w^{k}+\alpha_{k}(w^{k}-w^{k-1}).
\end{equation}
Again we introduce a new sequence $\{u^{k+1}\}_{k\geq1}$ by
\begin{equation}\label{E3.2}
u^{k+1}\in \partial F^{\ast}(-M^{\ast}x^{k}),
\end{equation}
 we can get
\begin{equation}\label{e3.10}
0=-\gamma Mu^{k+1}+x^{k}-2y^{k}+w^{k}+\alpha_{k}(w^{k}-w^{k-1})\quad and \quad -M^{\ast}x^{k}\in \partial F(u^{k+1}).
\end{equation}

From the first formula of (\ref{E3.1}) and (\ref{e3.10}), we obtain
\begin{equation}\label{e3.11}
x^{k}-y^{k} = \gamma Mu^{k+1} +\gamma Nv^{k} - \gamma b.
\end{equation}

Combining the second formula of (\ref{e3.10}) and (\ref{e3.11}), we have
\begin{align}
& 0\in \partial F(u^{k+1})+M^{\ast}x^{k};\nonumber\\
\Leftrightarrow &0\in \partial F(u^{k+1}) +M^{\ast}(y^{k}+\gamma Mu^{k+1}+\gamma Nv^{k}-\gamma b).\nonumber
\end{align}

Therefore, it is clear that
\begin{equation}\label{e3.12}
u^{k+1}=\arg\min_{u}\{F(u)+\langle y^{k},Mu\rangle+\frac{\gamma}{2}\|Mu+Nv^{k}-b\|^{2}\}.
\end{equation}
This is the first step of Algorithm \ref{algo2}.

Let $p^{k}=\alpha_{k}(w^{k}-w^{k-1})$, then the first formula of (\ref{E3.1}) can be rewritten as
\begin{equation}\label{E3.3}
w^{k}=y^{k}+\gamma b-p^{k}-\gamma Nv^{k},
\end{equation}
furthermore, we have
\begin{align}
p^{k+1}& =\alpha_{k+1}(w^{k+1}-w^{k});\nonumber\\
  &=\alpha_{k+1}(y^{k+1}+\gamma b-p^{k+1}-\gamma Nv^{k+1}-y^{k}-\gamma b+p^{k}+\gamma Nv^{k}),\nonumber
\end{align}
hence
\begin{equation}\label{E3.4}
p^{k+1}=\alpha_{k+1}(w^{k+1}-w^{k})=\frac{\alpha_{k+1}}{1+\alpha_{k+1}}(p^{k}+(y^{k+1}-y^{k})-\gamma N(v^{k+1}-v^{k})).
\end{equation}

From the third step of iteration scheme (\ref{algo1}), the first formula of (\ref{E3.1}), (\ref{e3.11}) and (\ref{E3.4}), we get
\begin{equation}\label{e3.13}
y^{k+1}=y^{k}+\alpha_{k+1}p^{k}+\gamma[N(v^{k+1}-v^{k})+(1+\alpha_{k+1})\lambda_{k}(Mu^{k+1}+ Nv^{k}- b)].
\end{equation}
This is the third step of Algorithm \ref{algo2}.

Substituting (\ref{e3.13}) into (\ref{E3.4}), we have
\begin{equation}\label{fourth-step}
p^{k+1}=\alpha_{k+1}[p^{k}+\gamma\lambda_{k}(Mu^{k+1}+Nv^{k}-b)].
\end{equation}
This is the fourth step of Algorithm \ref{algo2}.

Combining the second formula of (\ref{e3.10}) and (\ref{e3.13}), we obtain
\begin{align}
&0\in \partial G(v^{k+1})+N^{\ast}y^{k+1};\nonumber\\
&0\in \partial G(v^{k+1})+N^{\ast}(y^{k}+\alpha_{k+1}p^{k}+\gamma(N(v^{k+1}-v^{k})+(1+\alpha_{k+1})\lambda_{k}(Mu^{k+1}+ Nv^{k}- b))).\nonumber
\end{align}

Therefore, it is clear that
\begin{equation}\label{E3.5}
v^{k+1}=\arg\min_{v}\{ G(v)+\langle y^{k}+\alpha_{k+1}p^{k},Nv\rangle+\frac{\gamma}{2}\|N(v-v^{k})+(1+\alpha_{k+1})\lambda_{k}(Mu^{k+1}+ Nv^{k}- b)\|^{2}\}.
\end{equation}
This is the second step of Algorithm \ref{algo2}.

Therefore, Algorithm \ref{algo2} is equivalent to the inertial Douglas-Rachford splitting algorithm for the dual problem of the original problem (\ref{problem1}).

Next, we prove the convergence of Theorem \ref{theo2}. According to Theorem \ref{theo1}, there exists $\bar{w}\in H$ such that
\begin{equation}\label{e3.16}
w^{k}\rightharpoonup \bar{w} \quad as \quad k\rightarrow +\infty,
\end{equation}
\begin{equation}\label{e3.17}
w^{k+1}-w^{k}\rightarrow 0\quad as \quad k\rightarrow +\infty,
\end{equation}
\begin{equation}\label{e3.18}
y^{k}\rightharpoonup J_{\gamma B}\bar{w}\quad as \quad k\rightarrow +\infty,
\end{equation}
\begin{equation}\label{e3.19}
x^{k}\rightharpoonup J_{\gamma B}\bar{w}=J_{\gamma A}(2J_{\gamma B}\bar{w}-\bar{w})\quad as \quad k\rightarrow +\infty,
\end{equation}
\begin{equation}\label{e3.20}
y^{k}-x^{k}\rightarrow 0 \quad as \quad k\rightarrow +\infty.
\end{equation}

{\rm (i)}  From the first formula of (\ref{e3.10}), we have
\begin{align}
\gamma Mu^{k+1}=x^{k}-2y^{k}+w^{k}+\alpha_{k}(w^{k}-w^{k-1}),\nonumber
\end{align}
and from (\ref{e3.16}), (\ref{e3.17}), (\ref{e3.18}) and (\ref{e3.20}), we can get
\begin{equation}\label{eq3.21}
Mu^{k+1}\rightharpoonup \frac{1}{\gamma}(\bar{w}-J_{\gamma B}\bar{w}) \quad as \quad k\rightarrow +\infty.
\end{equation}
%

Moreover, by the first formula of (\ref{E3.1}), we have
\begin{align}
\gamma Nv^{k}=\gamma b+y^{k}-w^{k}-\alpha_{k}(w^{k}-w^{k-1}),\nonumber
\end{align}
and then through (\ref{e3.16}), (\ref{e3.17}) and (\ref{e3.18}), we derive that
\begin{equation}\label{eq3.23}
Nv^{k}\rightharpoonup \frac{1}{\gamma}(J_{\gamma B}\bar{w}-\bar{w})+b \quad as \quad k\rightarrow +\infty.
\end{equation}

Let
\begin{equation}\label{eq3.26}
Mu^{\ast}=\frac{1}{\gamma}(\bar{w}-J_{\gamma B}\bar{w})\ and \ Nv^{\ast}= \frac{1}{\gamma}(J_{\gamma B}\bar{w}-\bar{w})+b,
\end{equation}
because the operators $M$ and $N$ satisfy the condition $\|Mx\| \geq \theta_1 \|x\|$ and $\|Nx\| \geq \theta_2 \|x\|$ for all $x\in H$, according to (\ref{eq3.21}) and (\ref{eq3.23}), there exist $u^{\ast}$, $v^{\ast}$ such that
\begin{equation}\label{eq3.24}
u^{k}\rightharpoonup u^{\ast} \ and \ v^{k}\rightharpoonup v^{\ast} \quad as \quad k\rightarrow +\infty,
\end{equation}
and
\begin{equation}\label{eq3.25}
Mu^{\ast}+Nv^{\ast}=b.
\end{equation}

{\rm (ii)} It follows from (\ref{e3.11}) and (\ref{e3.20}) that
\begin{equation}\label{eq3.27}
Mu^{k+1}+Nv^{k}\rightarrow b \quad as \quad k\rightarrow +\infty
\end{equation}

{\rm (iii)} Let
\begin{equation}\label{eq3.28}
y^{\ast}=J_{\gamma B}\bar{w},
\end{equation}
from (\ref{e3.18}), we have
\begin{equation}\label{eq3.29}
y^{k}\rightharpoonup y^{\ast}.
\end{equation}

{\rm (iv)} According to (\ref{eq3.28}), we have
\begin{equation}\label{eq3.30}
\bar{w}\in y^{\ast}+\gamma By^{\ast},
\end{equation}
thus
\begin{align}\label{eq3.31}
Nv^{\ast}&=\frac{1}{\gamma}(J_{\gamma B}\bar{w}-\bar{w})+b;\nonumber\\
&=\frac{1}{\gamma}(y^{\ast}-\bar{w})+b;\nonumber\\
&\in -By^{\ast}+b.
\end{align}
Substitute $B=\partial(G^{\ast}\circ(-N^{\ast}))+b$ to (\ref{eq3.31}), we have
\begin{align}\label{eq3.32}
Nv^{\ast}\in N\partial G^{\ast}(-N^{\ast}y^{\ast})\Leftrightarrow 0 \in \partial G(v^{\ast})+N^{\ast}y^{\ast}.
\end{align}

Further, by (\ref{e3.19}) and (\ref{eq3.28}), we obtain
\begin{equation}\label{eq3.33}
y^{\ast}-\bar{w}\in \gamma Ay^{\ast},
\end{equation}
that is
\begin{align}\label{eq3.34}
Mu^{\ast}&=\frac{1}{\gamma}(\bar{w}-J_{\gamma B}\bar{w});\nonumber\\
&=\frac{1}{\gamma}(\bar{w}-y^{\ast});\nonumber\\
&\in -Ay^{\ast}.
\end{align}
Substitute $A=\partial(F^{\ast}\circ(-M^{\ast}))$ to (\ref{eq3.34}), we have
\begin{align}\label{eq3.35}
Mu^{\ast}\in M\partial F^{\ast}(-M^{\ast}y^{\ast})\Leftrightarrow 0 \in \partial F(u^{\ast})+M^{\ast}y^{\ast}.
\end{align}

According to (\ref{eq3.25}), (\ref{eq3.32}) and (\ref{eq3.35}), we prove that the existence point pair $(u^{\ast},v^{\ast},y^{\ast})$ satisfies the optimality condition (\ref{1.4}). Theorem \ref{theo2} (iv) can be obtained directly from Theorem \ref{theo1} (7).

{\rm (v)} We know that $F$ and $G$ are weakly lower semi-continuous (since $F$ and $G$ are convex) and therefore, from (i), we have
\begin{align}\label{e3.33}
\liminf_{k\rightarrow+\infty}(F(u^{k+1})+G(v^{k}))&\geq \liminf_{k\rightarrow+\infty} F(u^{k+1})+\liminf_{k\rightarrow+\infty} G(v^{k})\nonumber\\
&\geq F(u^{\ast})+G(v^{\ast})=v(P).
\end{align}

We deduce from the second formula of (\ref{E3.1}) and (\ref{e3.10}) that
\begin{equation}\label{e3.34}
G(v^{\ast})\geq G(v^{k})+\langle v^{\ast}-v^{k},-N^{\ast}y^{k}\rangle;
\end{equation}
\begin{equation}\label{e3.35}
F(u^{\ast})\geq F(u^{k+1})+\langle u^{\ast}-u^{k+1},-M^{\ast}x^{k}\rangle.
\end{equation}

Summing up (\ref{e3.34}) and (\ref{e3.35}), we obtain
\begin{equation}\label{e3.36}
v(P)\geq F(u^{k+1})+G(v^{k})+\langle u^{\ast}-u^{k+1},-M^{\ast}x^{k}\rangle+\langle v^{\ast}-v^{k},-N^{\ast}y^{k}\rangle,
\end{equation}
that is
\begin{align}\label{e3.37}
v(P)&\geq F(u^{k+1})+G(v^{k})+\langle Mu^{\ast}-Mu^{k+1},y^{k}-x^{k}\rangle\nonumber\\
&\quad\quad\quad\quad\quad\quad\quad+\langle Mu^{\ast}+Nv^{\ast}-Mu^{k+1}-Nv^{k},-y^{k}\rangle.
\end{align}

From (i), (ii), (iii), (\ref{e3.20}) and (\ref{eq3.25}), we have
\begin{equation}\label{e3.38}
\limsup_{k\rightarrow+\infty}(F(u^{k+1})+G(v^{k}))\leq v(P).
\end{equation}

Combining (\ref{e3.33}) and (\ref{e3.38}), we prove the first part of Theorem \ref{theo2} (v).

Again from the second formula of (\ref{E3.1}) and (\ref{e3.10}), we get
\begin{equation}\label{e3.39}
G(v^{k})+G^{\ast}(-N^{\ast}y^{k})=\langle -N^{\ast}y^{k},v^{k}\rangle;
\end{equation}
\begin{equation}\label{e3.40}
F(u^{k+1})+F^{\ast}(-M^{\ast}x^{k})=\langle -M^{\ast}x^{k},u^{k+1}\rangle.
\end{equation}

Summing up (\ref{e3.39}) and (\ref{e3.40}), we have
\begin{align}\label{e3.41}
F(u^{k+1})+G(v^{k})&=-F^{\ast}(-M^{\ast}x^{k}) -G^{\ast}(-N^{\ast}y^{k})+\langle -x^{k},Mu^{k+1}\rangle+\langle -y^{k},Nv^{k}\rangle;\nonumber\\
&=-F^{\ast}(-M^{\ast}x^{k}) -G^{\ast}(-N^{\ast}y^{k})+\langle y^{k}-x^{k},Mu^{k+1}\rangle+\langle -y^{k},Mu^{k+1}+Nv^{k}\rangle.\nonumber
\end{align}

Finally, taking into account (i), (ii), (\ref{e3.20}) and the first part of Theorem \ref{theo2} (v), we obtain
\begin{equation}\label{e3.42}
\lim_{k\rightarrow+\infty}(-F^{\ast}(-M^{\ast}x^{k}) -G^{\ast}(-N^{\ast}y^{k})-\langle y^{k},b\rangle)=v(P)=v(D).
\end{equation}
This completes the proof.
\end{proof}

\begin{theorem}\label{theo3.2}
Assuming (\ref{problem1}) has an optimal solution and the condition (\ref{3.1}) is satisfied. Let the bounded linear operators $M$ and $N$ satisfy the condition that $\exists \theta_1 >0, \theta_2 >0$ such that $\|Mx\| \geq \theta_1 \|x\|$ and $\|Nx\| \geq \theta_2 \|x\|$, for all $x\in H$. Consider the sequence generated by Algorithm \ref{algo2}. Let $0\leq\alpha_{k}\leq\alpha<1$ and $0<\underline{\lambda}\leq\lambda_{k}\leq\overline{\lambda}<2$. Let $\sum_{k=1}^{+\infty}\alpha_{k+1}\|p^{k}+\gamma\lambda_{k}(Mu^{k+1}+Nv^{k}-b)\|^{2}<+\infty$. Then there exists a point pair $(u^{\ast},v^{\ast},y^{\ast})$, which is the saddle point of Lagrange function, where $(u^{\ast},v^{\ast})$ is the optimal solution of (\ref{problem1}), $y^{\ast}$ is the optimal solution of (\ref{1.8}), and $v(P)=v(D)$. The following statements are true:
\vskip 1mm
\rm (i)\, $(u^{k},v^{k})_{k\geq1}$ converges weakly to $(u^{\ast},v^{\ast})$;

\rm (ii)\, $(Mu^{k+1}+Nv^{k})_{k\geq1}$ converges strongly to $b$;

\rm (iii)\, $(y^{k})_{k\geq1}$ converges weakly to $y^{\ast}$;

\rm (iv)\, if $F^{\ast}$ or $G^{\ast}$ is uniformly convex, the $(y^{k})_{k\geq1}$ converges strongly to unique optimal solution of (D);

\rm (v)\, $\lim_{k\rightarrow+\infty}F(u^{k+1})+G(v^{k})=v(P)=v(D)= \lim_{k\rightarrow+\infty}(-F^{\ast}(-M^{\ast}x^{k})-G^{\ast}(-N^{\ast}y^{k})-\langle y^{k},b\rangle)$,where the sequence $(x^{k})_{k\geq1}$ is defined by
$$
x^{k}=y^{k}+\gamma Mu^{k+1}+\gamma Nv^{k}-\gamma b,
$$
and $(x^{k})_{k\geq1}$ converges weakly to $y^{\ast}$.
\end{theorem}
\begin{proof}
The proof of Theorem \ref{theo3.2} is similar to Theorem \ref{theo2}, so we omit it here.
\end{proof}

\begin{remark}
Notice that, in finite-dimensional case, the assumption on $M$ and $N$ in Theorem \ref{theo2} and Theorem \ref{theo3.2} means that $M$ and $N$ are matrices with full column rank.
\end{remark}

\begin{remark}\label{rema3.1}
As we can see, when $\alpha_{k}\equiv 0$, the iterative sequences of Algorithm \ref{algo2} reduces to the GADMM (\ref{GADMM}).

Suppose that the matrix $N$ is full column rank, let $b-Nv=z$, then $Mu=z$ and $v=(N^{\ast}N)^{-1}N^{\ast}(b-z)$. Let $H(z)=G((N^{\ast}N)^{-1}N^{\ast}(b-z))=G(v)$, then problem (\ref{problem1}) is equivalent to problem (\ref{problem22}), i.e., $\min_{u\in H_{1}}F(u)+H(Mu)$.
Therefore, we can directly apply the algorithm (\ref{inertial-ADMM-bot}) and obtain the following iteration scheme:
\begin{equation}\label{algo2-equivalent-case}
 \left\{
\begin{aligned}
 & u^{k+1} =\arg\min_{u}\{ F(u)+\langle y^{k}-\alpha_{k}(y^{k}-y^{k-1}-\gamma N(v^{k}-v^{k-1})),Mu\rangle +\frac{\gamma}{2}\|Mu+Nv^{k}-b\|^{2}\},\\
 & \bar{v}^{k+1}=-\frac{\alpha_{k+1}}{\gamma}(N^{\ast}N)^{-1}N^{\ast}[\gamma \lambda_{k}(Mu^{k+1}+Nv^{k}-b)+(1-\lambda_{k})\alpha_{k}(y^{k}-y^{k-1}-\gamma N(v^{k}-v^{k-1}))],\\
 & v^{k+1} =\arg\min_{v}\{G(v+\bar{v}^{k+1})+\langle y^{k}+(1-\lambda_{k})\alpha_{k}(y^{k}-y^{k-1}-\gamma N(v^{k}-v^{k-1})),Nv\rangle \\
 & \quad \quad \quad \quad \quad \quad \quad \quad \quad \quad+\frac{\gamma}{2}\|N(v-v^{k})+\lambda_{k}(Mu^{k+1}+Nv^{k}-b)\|^{2}\},\\
 & y^{k+1}=y^{k}+\gamma[N(v^{k+1}-v^{k})+\lambda_{k}(Mu^{k+1}+Nv^{k}-b)] +(1-\lambda_{k})\alpha_{k}(y^{k}-y^{k-1}-\gamma N(v^{k}-v^{k-1})).
\end{aligned}
\right.
\end{equation}

In the following, we prove that Algorithm \ref{algo2} is equivalent to (\ref{algo2-equivalent-case}). In fact, according to (\ref{E3.1}) and (\ref{e3.3}), we have $\gamma Nv^{k}=\gamma b+y^{k}-w^{k}-\alpha_{k}(w^{k}-w^{k-1})$ and  $\gamma Nv^{k}=y^{k}+\gamma b -w^{k}-p^{k}$, respectively. Let $\gamma Nt^{k} = \gamma b+y^{k}-w^{k}$ and $\gamma N\bar{t}^{k} = -\alpha_{k}(w^{k}-w^{k-1})$, we obtain
\begin{equation}\label{R1}
\gamma Nv^{k}=\gamma Nt^{k}+\gamma N\bar{t}^{k},\,\  p^{k}=-\gamma N\bar{t}^{k} \,\ and \,\ \gamma N\bar{t}^{k} = -\alpha_{k}(y^{k}-y^{k-1}-\gamma N(t^{k}-t^{k-1})).
\end{equation}

Substituting the first formula of (\ref{R1}) into step 1 of Algorithm \ref{algo2}, we obtain
\begin{equation}\label{R2}
u^{k+1}=\arg\min_{u}\{ F(u)+\langle y^{k},Mu\rangle+\frac{\gamma}{2}\|Mu+N(t^{k}+\bar{t}^{k})-b\|^{2}\},
\end{equation}
and then substitute the last formula of (\ref{R1}) into (\ref{R2}), we get
\begin{equation}\label{R3}
u^{k+1}=\arg\min_{u}\{ F(u)+\langle y^{k}-\alpha_{k}(y^{k}-y^{k-1}-\gamma N(t^{k}-t^{k-1})),Mu\rangle+\frac{\gamma}{2}\|Mu+Nt^{k}-b\|^{2}\}.
\end{equation}

Substituting the (\ref{R1}) into step 4 of Algorithm \ref{algo2}, it is easy to get
\begin{equation}\label{R4}
\begin{aligned}
\bar{t}^{k+1}& =-\frac{1}{\gamma}(N^{\ast}N)^{-1}N^{\ast}p^{k+1} \\
&= -\frac{\alpha_{k+1}}{\gamma}(N^{\ast}N)^{-1}N^{\ast}(-\gamma N\bar{t}^{k}+\gamma\lambda_{k}(Mu^{k+1}+N(t^{k}+\bar{t}^{k})-b)) \\
&= -\frac{\alpha_{k+1}}{\gamma}(N^{\ast}N)^{-1}N^{\ast}[\gamma \lambda_{k}(Mu^{k+1}+Nt^{k}-b)+(1-\lambda_{k})\alpha_{k}(y^{k}-y^{k-1}-\gamma N(t^{k}-t^{k-1}))].
\end{aligned}
\end{equation}

Similarly, substituting the (\ref{R1}) into step 2 of Algorithm \ref{algo2}, we have
\begin{align}\label{R5}
t^{k+1} =\arg&\min_{t}\{G(t+\bar{t}^{k+1})+\langle y^{k}-\gamma\alpha_{k+1}N\bar{t}^{k},Nt\rangle \\ &+\frac{\gamma}{2}\|N(t+\bar{t}^{k+1}-t^{k}-\bar{t}^{k}) +(1+\alpha_{k+1})\lambda_{k}(Mu^{k+1}+N(t^{k}+\bar{t}^{k})-b)\|^{2}\},
\end{align}
which implies that
\begin{align}\label{R6}
t^{k+1} =\arg\min_{t}\{G(t+\bar{t}^{k+1})&+\langle y^{k}+(1-\lambda_{k})\alpha_{k}(y^{k}-y^{k-1}-\gamma N(t^{k}-t^{k-1})),Nt\rangle \\
 & +\frac{\gamma}{2}\|N(t-t^{k})+\lambda_{k}(Mu^{k+1}+Nt^{k}-b)\|^{2}\}.
\end{align}

Finally, combining steps 3 of Algorithm \ref{algo2} and (\ref{R1}), we obtain
\begin{align}\label{R7}
y^{k+1}=y^{k}+\gamma[N(t^{k+1}-t^{k})+\lambda_{k}(Mu^{k+1}+Nt^{k}-b)] +(1-\lambda_{k})\alpha_{k}(y^{k}-y^{k-1}-\gamma N(t^{k}-t^{k-1})).
\end{align}

Compared with (\ref{algo2-equivalent-case}), it is obvious that the iteration scheme of Algorithm \ref{algo2} is more concise.
\end{remark}
\vskip 2mm

\section{Numerical experiments}
\label{numer_test}

In this section, we carry out simulation experiments and compare the proposed algorithm (Algorithm \ref{algo2}) with other state-of-the-art algorithms include the classical ADMM \cite{Yuanx2013PJO}, GADMM \cite{Eckstein1992}, the inertial ADMM of Chen et al. \cite{Chen2015} (iADMM\_Chen). All the experiments are conducted on 64-bit Windows 10 operating system with an Inter Core(TM) i5-6500 CPU and 8GB memory. All the codes are tested in MATLAB R2016a.

\subsection{Robust principal component pursuit (RPCP) problem}

The robust principal component analysis (RPCA) problem was first introduced by Cand\`{e}s et al. \cite{Candes2009ACM}, which can be formulated as the following optimization model:
\begin{equation}
\begin{aligned}\label{e4.1}
& \min_{u,v}\quad rank(u)+\mu \|v\|_{0} \\
& s.t. \quad u+v=b.
\end{aligned}
\end{equation}
The objective function in (\ref{e4.1}) includes the rank of the matrix $u$ and the $\ell_{0}$-norm of matrix $v$, and $\mu>0$ is the penalty parameter balancing the low rank and sparsity. The RPCA (\ref{e4.1}) seeks to decompose a matrix $b$ into two parts: one is low-rank and the other is sparse. It has wide applications in image and video processing and many other fields. We refer interested readers to \cite{Bouwmans2014,Masq2018} for a comprehensive review of RPCA and its variants.  It is known that the original RPCA (\ref{e4.1}) is NP-hard. By using convex relaxation technique, the rank function of the matrix is usually replaced by the nuclear norm of the matrix, and the $\ell_{0}$-norm of the matrix is replaced by the $\ell_{1}$-norm of the matrix. Therefore, we can get the following convex optimization model:
\begin{equation}
\begin{aligned}\label{e4.2}
& \min_{u,v} \quad \|u\|_{\ast}+\mu\|v\|_{1} \\
& s.t. \quad u+v=b,
\end{aligned}
\end{equation}
where $\|u\|_{\ast}=\sum^{n}_{k=1}\sigma_{k}(u)$ is the nuclear norm of the matrix and $\sigma_{k}(u)$ represents the $k$ singular value of the matrix, and $\|v\|_{1} = \sum_{ij}|v_{ij}|$. Under certain conditions, problem (\ref{e4.2}) is equivalent to (\ref{e4.1}). See for example \cite{Candes2009ACM,Chandrasekaran2011SIAM}. The optimization problem is usually called robust principal component pursuit (RPCP). In fact,  let $F(u)=\|u\|_{\ast}$, $G(v)=\mu\|v\|_{1}$ and $M=N=I$. Then the RPCP (\ref{e4.2}) is a special case of the general problem (\ref{problem1}). Therefore, the classical ADMM algorithm, GADMM algorithm and inertial ADMM algorithms (includes iADMM\_Chen and Algorithm \ref{algo2}) can be used to solve the convex optimization problem (\ref{e4.2}).

We follow \cite{Candes2009ACM} to generate the simulation data. In the experiment, a low rank matrix is randomly generated by the following method. Firstly, two long strip matrices $L=randn(m,r)$ and $R=randn(m,r)$ are randomly generated, and then $u^{\ast}=LR^{T}$ is calculated, where $m$ and $r$ are the order and rank of matrix $u^{\ast}$, respectively. At the same time, a sparse matrix $v^{\ast}$ with uniform distribution of non-zero elements and uniform distribution of values between $[-500,500]$ is generated. Finally, the target matrix is generated by $b=u^{\ast}+v^{\ast}$.

\subsection{Parameters setting}

In this part, we show how to choose parameters for the studied algorithms. Let $\mu=1/\sqrt{m}$, where $m$ is the order of the matrix. For the common parameter $\gamma$ of classical ADMM, GADMM, iADMM\_Chen and Algorithm \ref{algo2}, we take the value of $0.01$. We show the relationship between the experimental results of iADMM\_Chen and Algorithm \ref{algo2} with the first condition (Algorithm \ref{algo2}-1) and the inertia parameter $\alpha_{k}$. Subsequent parameter selection will be based on this. For the private parameters of the Algorithm \ref{algo2}-1, we fix $\sigma= 0.01$, $\delta$ and relaxation parameter $\lambda_{k}$ are
\begin{equation}\label{e4.3}
\delta=1+\frac{\alpha^{2}(1+\alpha)+\alpha\sigma}{1-\alpha^{2}}\quad and\quad \lambda_{k}=2\frac{\delta-\alpha[\alpha(1+\alpha)+\alpha\delta+\sigma]}{\delta[1+\alpha(1+\alpha)+\alpha\delta+\sigma]},
\end{equation}
where $\alpha$ takes four different values 0.05, 0.1, 0.2 and 0.3, respectively, and let inertia parameter $\alpha_{k}=\alpha$. And the inertia parameter $\alpha_{k}$ of iADMM\_Chen are the same as that of Algorithm \ref{algo2}-1. The different selection of inertia parameters are listed in Table \ref{table1}. In the experiment, the order of objective matrix $b$ is $m=1000$, the rank of low rank matrix $u^{\ast}$ is $r=0.1m$, and the sparsity of sparse matrix $v^{\ast}$ is $\|v^{\ast}\|_{0}=0.05m^{2}$, respectively.
\begin{table}[htbp]
\footnotesize
\centering
\caption{Parameters selection of the iADMM\_Chen and Algorithm \ref{algo2}-1}
\begin{tabular}{|c|c|c|c|}
\hline
 Methods & $\gamma$  &  Inertia parameter $\alpha_{k}$ &  Relaxation parameter $\lambda_{k}$ \\
\hline
 \multirow{4}[1]{*}{Algorithm \ref{algo2}-1} &\multirow{4}[1]{*}{$0.01$} & $0.05$ & $1.7874$  \\ \cline{3-4}
  & & $0.1$ & $1.6019$ \\  \cline{3-4}
  & & $0.2$ & $1.2496$ \\  \cline{3-4}
  & & $0.3$ & $0.9243$ \\
  \hline
 \multirow{4}[1]{*}{iADMM\_Chen} & \multirow{4}[1]{*}{$0.01$} & $0.05$ & None \\
  \cline{3-4}
  & & $0.1$ & None \\ \cline{3-4}
  & & $0.2$ & None \\ \cline{3-4}
  & & $0.3$ & None \\
\hline
\end{tabular}\label{table1}
\end{table}

We define the relative error $rel\ u$, $rel\ v$ and $rel\ b$ as the stopping criterion, i.e.,
\begin{align}\label{e4.4}
rel\, u:=\frac{\|u^{k+1}-u^{k}\|_{F}}{\|u^{k}\|_{F}}&,\ rel\, v:=\frac{\|v^{k+1}-v^{k}\|_{F}}{\|v^{k}\|_{F}},\ rel\, b:=\frac{\|b^{k+1}-b^{k}\|_{F}}{\|b^{k}\|_{F}},\nonumber\\
& \max(rel\,u , rel\,v , rel\,b)\leq \varepsilon \nonumber
\end{align}
where $\varepsilon$ is a given small constant.

\begin{table}[htbp]
\footnotesize
\centering
\caption{Numerical experimental results of iADMM\_Chen and Algorithm \ref{algo2}-1 under different inertia parameters $\alpha_{k}$($rel\ u^{\ast}$ and $rel\ v^{\ast}$ are defined as $\frac{\|u^{k}-u^{\ast}\|_{F}}{\|u^{\ast}\|_{F}}$ and $\frac{\|v^{k}-v^{\ast}\|_{F}}{\|v^{\ast}\|_{F}}$, respectively).}
\begin{tabular}{c|c|ccccccccc}
\hline
\multirow{2}[1]{*}{Methods} & \multirow{2}[1]{*}{$\alpha_{k}$} &\multicolumn{4}{c}{$\varepsilon=1e-5$} & & \multicolumn{4}{c}{$\varepsilon=1e-7$} \\
 \cline{3-6} \cline{8-11}
& & $k$ & $rel\ u^{\ast}$ & $rel\ v^{\ast}$ & $rank(u^{\ast})$ & & $k$ & $rel\ u^{\ast}$ & $rel\ v^{\ast}$ & $rank(u^{\ast})$ \\
\hline
\multirow{4}[1]{*}{Algorithm \ref{algo2}-1} & 0.05 & 70 & 5.3242e-6 & 1.2665e-6 & 100 & & 101 & 1.7586e-6 & 6.0458e-7 & 100 \\
\cline{2-11}
& 0.1 & 56 & 6.8826e-6 & 1.8277e-6 & 100 & & 96 & 1.7588e-6 & 6.0457e-7 & 100 \\
\cline{2-11}
& 0.2 & 49 & 1.3062e-5 & 4.1950e-6 & 100 & & 69 & 4.8588e-6 & 1.7194e-6 & 100 \\
\cline{2-11}
& 0.3 & 53 & 1.7610e-5 & 5.6121e-6 & 100 & & 77 & 4.8727e-6 & 1.7199e-6 & 100 \\
\hline
\multirow{4}[1]{*}{iADMM\_Chen} & 0.05 & 55 & 3.1465e-5 & 9.6907e-6 & 100 & & 91 & 4.8607e-6 & 1.7197e-6 & 100 \\
\cline{2-11}
& 0.1 & 53 & 2.9038e-5 & 9.6133e-6 & 100 & & 87 & 4.8698e-6 & 1.7199e-6 & 100 \\
\cline{2-11}
& 0.2 & 49 & 2.8558e-5 & 9.1194e-6 & 100 & & 78 & 4.8658e-6 & 1.7196e-6 & 100 \\
\cline{2-11}
& 0.3 & 50 & 1.7318e-5 & 5.6383e-6 & 100 & & 72 & 4.8530e-6 & 1.7192e-6 & 100 \\
\hline
\end{tabular}\label{table2}
\end{table}

From Table \ref{table2}, we can see that when the inertia parameters $\alpha_{k}$ of iADMM\_Chen and Algorithm \ref{algo2}-1 are $0.3$ and $0.2$, respectively, the experimental results are the best. In the following experiments, we fix the inertia parameters $\alpha_{k}$ of iADMM\_Chen and Algorithm \ref{algo2}-1 to $0.3$ and $0.2$, respectively. The parameters of classical ADMM, GADMM, and Algorithm \ref{algo2} with the second condition (Algorithm \ref{algo2}-2) are defined as follows in Table \ref{table4}.

\begin{table}[htbp]
\small
\centering
\caption{Parameters selection of the compared iterative algorithms.}
\begin{tabular}{c|c|c|c}
\hline
 Methods & $\gamma$  & $\lambda_{k}$ &  Inertia parameter $\alpha_{k}$ \\
 \hline
 \hline
 ADMM  & \multirow{5}[1]{*}{$0.01$} & None & None \\
 GADMM &  & 1.6 & None \\
 iADMM\_Chen  &  & None & 0.3 \\
 Algorithm \ref{algo2}-1 &  & 1.2496 & 0.2 \\
 Algorithm \ref{algo2}-2 & & 1.5 & $\min\{\frac{1}{k^{2}\|p^{k}+\gamma\lambda_{k}(Mu^{k+1} + Nv^{k} -b)\|^{2}},0.05\}$ \\
\hline
\end{tabular}\label{table4}
\end{table}

\subsection{Results and discussions}

In the experiment, the order of objective matrix $b$ is $m=500$, $m=800$ and $m=1000$, the rank of low rank matrix $u^{\ast}$ is $r=0.05m$ and $r=0.1m$, and the sparsity of sparse matrix $v^{\ast}$ is $\|v^{\ast}\|_{0}=0.05m^{2}$ and $\|v^{\ast}\|_{0}=0.1m^{2}$, respectively.

Table \ref{table3} is a comparison of the numerical experimental results of the classical ADMM, GADMM, iADMM\_Chen and Algorithm \ref{algo2} (include Algorithm \ref{algo2}-1 and Algorithm \ref{algo2}-2). We conclude from Table \ref{table3} that the two inertial ADMMs (iADMM\_Chen and Algorithm \ref{algo2}) and the GADMM are better than the classical ADMM in terms of iteration numbers and accuracy. The iADMM\_Chen and Algorithm \ref{algo2} are similar in  terms of accuracy.  The proposed Algorithm \ref{algo2} is  better than iADMM\_Chen in the number of iterations in most cases. Besides, the proposed Algorithm \ref{algo2} is also comparable with the GADMM (\ref{GADMM}). In some cases, the number of iteration of the GADMM is higher than that of Algorithm \ref{algo2}. For example, when $m=800$, $rank(u^{\ast})=0.05m$ and $\|v^{\ast}\|_{0}=0.05m^{2}$, and $m=1000$, $rank(u^{\ast})=0.05m$ and $\|v^{\ast}\|_{0}=0.1m^{2}$ and so on. But, in most cases, the number of iteration of the GADMM is less than Algorithm \ref{algo2}.

\begin{table}[htbp]
\tiny
\newcommand{\tabincell}[2]{\begin{tabular}{@{}#1@{}}#2\end{tabular}}  
\centering
\caption{Comparison of numerical experimental results of ADMM, GADMM, iADMM\_Chen and Algorithm \ref{algo2}
($rel\ u^{\ast}$ and $rel\ v^{\ast}$ are defined as $\frac{\|u^{k}-u^{\ast}\|_{F}}{\|u^{\ast}\|_{F}}$ and $\frac{\|v^{k}-v^{\ast}\|_{F}}{\|v^{\ast}\|_{F}}$, respectively).}
\begin{tabular}{c|c|ccccc}
\hline
  & $m$ & Methods & $k$ & $rel\ u^{\ast}$ & $rel\ v^{\ast}$ & $rank(u^{k})$  \\
\hline
  \multirow{15}[1]{*}{\tabincell{c}{$rank(u^{\ast})=0.05m$  \\ $\|v^{\ast}\|_{0}=0.05m^{2}$ \\ $\varepsilon=1e-7$  }} & \multirow{5}[1]{*}{500 } & ADMM & 58 & 1.6323e-5 & 3.6376e-6 & 25 \\ \cline{3-7}
  & &GADMM & 45 & 1.6199e-5 & 3.6358e-6 & 25 \\ \cline{3-7}
  & & iADMM\_Chen & 46 & 1.6150e-5 & 3.6351e-6 & 25 \\ \cline{3-7}
  & & Algorithm \ref{algo2}-1 & 48 & 1.6153e-5 & 3.6351e-6 &  25 \\ \cline{3-7}
  & & Algorithm \ref{algo2}-2 & 45 & 1.6151e-5 & 3.6363e-6 &  25 \\ \cline{2-7}
  & \multirow{5}[1]{*}{800} & ADMM & 89 & 5.3299e-6 & 1.5083e-6 & 40 \\ \cline{3-7}
  & &GADMM & 76 & 2.3394e-6 & 6.5742e-7 & 40 \\ \cline{3-7}
  & & iADMM\_Chen & 68 & 5.3209e-6 & 1.5081e-6 &  40 \\ \cline{3-7}
  & & Algorithm \ref{algo2}-1 & 63 & 5.3177e-6 & 1.5081e-6 & 40 \\ \cline{3-7}
  & & Algorithm \ref{algo2}-2 & 64 & 5.3180e-6 & 1.5081e-6 & 40 \\ \cline{2-7}
  & \multirow{5}[1]{*}{1000} & ADMM & 76 & 6.3916e-6 & 2.0164e-6 & 50 \\ \cline{3-7}
  & &GADMM & 57 & 6.3510e-6 & 2.0155e-6 & 50 \\ \cline{3-7}
  & & iADMM\_Chen & 57 & 6.3514e-6 & 2.0155e-6 & 50 \\ \cline{3-7}
  & & Algorithm \ref{algo2}-1 & 55 & 6.3540e-6 & 2.0156e-6 &  50  \\ \cline{3-7}
  & & Algorithm \ref{algo2}-2 & 57 & 6.3568e-6 & 2.0156e-6 &  50 \\
\hline
  \multirow{15}[1]{*}{\tabincell{c}{$rank(u^{\ast})=0.05m$ \\ $\|v^{\ast}\|_{0}=0.1m^{2}$ \\ $\varepsilon=1e-7$ }} & \multirow{5}[1]{*}{500 } & ADMM & 89 & 1.7912e-5 & 2.8058e-6 & 25 \\ \cline{3-7}
  & &GADMM & 62 & 1.7661e-5 & 2.8031e-6 & 25 \\ \cline{3-7}
  & & iADMM\_Chen & 68 & 1.7626e-5 & 2.8027e-6 & 25\\ \cline{3-7}
  & & Algorithm \ref{algo2}-1 & 64 & 1.7658e-5 & 2.8031e-6 &  25 \\ \cline{3-7}
  & & Algorithm \ref{algo2}-2 & 65 & 1.7662e-5 & 2.8031e-6 &  25 \\ \cline{2-7}
  & \multirow{5}[1]{*}{800} & ADMM & 118 & 8.9767e-6 & 1.7795e-6 & 40 \\ \cline{3-7}
  & &GADMM & 90 & 4.9539e-6 & 9.9398e-7 & 40 \\ \cline{3-7}
  & & iADMM\_Chen & 94 & 4.9658e-6 & 9.9416e-7 &  40 \\ \cline{3-7}
  & & Algorithm \ref{algo2}-1 & 92 & 4.9539e-6 & 9.9398e-7 & 40 \\ \cline{3-7}
  & & Algorithm \ref{algo2}-2 & 93 & 4.9641e-6 & 9.9412e-7 & 40 \\ \cline{2-7}
  & \multirow{5}[1]{*}{1000} & ADMM & 127 & 5.6708e-6 & 1.1938e-6 & 50 \\ \cline{3-7}
  & &GADMM & 109 & 3.4869e-6 & 7.5976e-7 & 50 \\ \cline{3-7}
  & & iADMM\_Chen & 101 & 4.4912e-6 & 9.8826e-7 &  50\\ \cline{3-7}
  & & Algorithm \ref{algo2}-1 & 98 & 4.4909e-6 & 9.8826e-7 &   50 \\ \cline{3-7}
  & & Algorithm \ref{algo2}-2 & 98 & 4.4898e-6 & 9.8825e-7 & 50 \\
\hline
  \multirow{15}[1]{*}{\tabincell{c}{$rank(u^{\ast})=0.1m$ \\ $\|v^{\ast}\|_{0}=0.05m^{2}$ \\ $\varepsilon=1e-7$ }} & \multirow{5}[1]{*}{500 } & ADMM & 68 & 7.8951e-6 & 1.7185e-6 & 50 \\ \cline{3-7}
  & &GADMM & 49 & 7.8838e-6 & 1.7181e-6 & 50 \\ \cline{3-7}
  & & iADMM\_Chen & 54 & 7.8842e-6 & 1.7182e-6 & 50\\ \cline{3-7}
  & & Algorithm \ref{algo2}-1 & 51 & 7.8840e-6 & 1.7181e-6 &  50 \\ \cline{3-7}
  & & Algorithm \ref{algo2}-2 & 50 & 7.8838e-6 & 1.7181e-6 & 50 \\ \cline{2-7}
  & \multirow{5}[1]{*}{800} & ADMM & 78 & 7.4927e-6 & 2.1961e-6 &  80\\ \cline{3-7}
  & &GADMM & 60 & 4.3186e-6 & 1.2655e-6 & 80 \\ \cline{3-7}
  & & iADMM\_Chen & 68 & 4.3188e-6 & 1.2655e-6 &   80\\ \cline{3-7}
  & & Algorithm \ref{algo2}-1 & 64 & 4.3262e-6 & 1.2657e-6 & 80 \\ \cline{3-7}
  & & Algorithm \ref{algo2}-2 & 64 & 4.3187e-6 & 1.2655e-6 & 80 \\ \cline{2-7}
  & \multirow{5}[1]{*}{1000} & ADMM & 95 & 5.4043e-6 & 1.8768e-6 &  100\\ \cline{3-7}
  & &GADMM & 69 & 4.8821e-6 & 1.7201e-6 & 100 \\ \cline{3-7}
  & & iADMM\_Chen & 72 & 4.8530e-6 & 1.7192e-6 &  100\\ \cline{3-7}
  & & Algorithm \ref{algo2}-1 & 69 & 4.8588e-6 & 1.7194e-6 &    100\\ \cline{3-7}
  & & Algorithm \ref{algo2}-2 & 70 & 4.8735e-6 & 1.7200e-6 & 100 \\
\hline
  \multirow{15}[1]{*}{\tabincell{c}{$rank(u^{\ast})=0.1m$ \\ $\|v^{\ast}\|_{0}=0.1m^{2}$ \\ $\varepsilon=1e-7$ }} & \multirow{5}[1]{*}{500 } & ADMM & 104 & 8.1931e-6 & 1.2617e-6 & 50 \\ \cline{3-7}
  & &GADMM & 84 & 6.3919e-6 & 9.9343e-7 & 50 \\ \cline{3-7}
  & & iADMM\_Chen & 86 & 6.3780e-6 & 9.9311e-7 & 50\\ \cline{3-7}
  & & Algorithm \ref{algo2}-1 & 88 & 6.3918e-6 & 9.9343e-7 &  50 \\ \cline{3-7}
  & & Algorithm \ref{algo2}-2 & 76 & 8.2038e-6 & 1.2620e-6 & 50 \\ \cline{2-7}
  & \multirow{5}[1]{*}{800} & ADMM & 138 & 7.3081e-6 & 1.4414e-6 & 80 \\ \cline{3-7}
  & &GADMM & 103 & 2.1651e-6 & 4.6292e-7 & 80 \\ \cline{3-7}
  & & iADMM\_Chen & 113 & 2.1650e-6 & 4.6292e-7 &  80 \\ \cline{3-7}
  & & Algorithm \ref{algo2}-1 & 107 & 2.1652e-6 & 4.6292e-7 &  80\\ \cline{3-7}
  & & Algorithm \ref{algo2}-2 & 107 & 2.1690e-6 & 4.6302e-7 & 80 \\ \cline{2-7}
  & \multirow{5}[1]{*}{1000} & ADMM & 140 & 5.0963e-6 & 1.1111e-6 & 100 \\ \cline{3-7}
  & &GADMM & 103 & 1.8990e-6 & 4.3722e-7 & 100 \\ \cline{3-7}
  & & iADMM\_Chen & 99 & 5.0276e-6 & 1.1086e-6 &  100\\ \cline{3-7}
  & & Algorithm \ref{algo2}-1 & 104 & 1.9001e-6 & 4.3725e-7 &  100  \\ \cline{3-7}
  & & Algorithm \ref{algo2}-2 & 107 & 1.8990e-6 & 4.3722e-7 & 100 \\
\hline
\end{tabular}\label{table3}
\end{table}

\section{Conclusions}

The ADMM is a popular method for solving many structural convex optimization problems.
In this paper, we proposed an inertial ADMM for solving the two-block separable convex optimization problem with linear equality constraints (\ref{problem1}), which derived from the inertial Douglas-Rachford splitting algorithm applied to the dual of (\ref{problem1}). The obtained algorithm generalized the inertial ADMM of \cite{Bot2016MTA}. Furthermore, we proved the convergence results of the proposed algorithm under mild conditions on the parameters. Numerical experiments for solving the RPCP (\ref{e4.2}) showed that the advantage of the proposed algorithm over existing iterative algorithms including the classical ADMM and the inertial ADMM introduced by Chen et al. \cite{Chen2015}. We also found the proposed algorithm  is  comparable with the GADMM (\ref{GADMM}).

\section*{Acknowledgement}
This work was funded by the National Natural Science Foundation of China (11661056, 11771198, 11771347, 91730306, 41390454, 11401293), the China Postdoctoral Science Foundation (2015M571989)
and the Jiangxi Province Postdoctoral Science Foundation (2015KY51).


\begin{thebibliography}{10}

\bibitem{Glowinski1975}
R.~Glowinski and A.~Marroco.
\newblock Sur l'approximation, par elements finis d'ordre un, et la resolution,
  par penalisation-dualite, d'une classe de problemes de dirichlet non
  lineares.
\newblock {\em Revue Francaise d'Automatique, Informatique et Recherche
  Operationelle}, 9:41--76, 1975.

\bibitem{Gabay1976}
D.~Gabay and B.~Mercier.
\newblock A dual algorithm for the solution of nonlinear variational problems
  via finite element approximation.
\newblock {\em Comput. Math. Appl.}, 2:17--40, 1976.

\bibitem{Eckstein1992}
J.~Eckstein and D.~Bertsekas.
\newblock On the douglas-rachford splitting method and the proximal point
  algorithm for maximal monotone operators.
\newblock {\em Math. Program.}, 55(1):293--318, 1992.

\bibitem{goldstein2009}
T.~Goldstein and S.~Osher.
\newblock The split bregman method for $\ell_1$-regularized problems.
\newblock {\em SIAM J Imaging Sci}, 2:323--343, 2009.

\bibitem{esser2009}
E.~Esser.
\newblock Applications of lagrangian-based alternating direction methods and
  connections to split bregman.
\newblock CAM Report 09-31, UCLA, April 2009.

\bibitem{boyd1}
S.~Boyd, N.~Parikh, E.~Chu, B.~Peleato, and J.~Eckstein.
\newblock Distrituted optimization and statistical learning via the alternating
  direction method of multipliers.
\newblock {\em Found. Trends Mach. Learn.}, 3:1--122, 2010.

\bibitem{He2012SIAMNA}
B.~He and X.M. Yuan.
\newblock On the o(1/n) convergence rate of the douglas-rachford alternating
  direction method.
\newblock {\em SIAM J. Numer. Anal.}, 50(2):700--709, 2012.

\bibitem{Monteiro2013SIAMJO}
R.D. Monteiro and B.F. Svaiter.
\newblock Iteration-complexity of block-decomposition algoirhms and the
  alternating direction method of multipliers.
\newblock {\em SIAM J. Optim.}, 23(1):475--507, 2013.

\bibitem{Fang2015MPC}
E.X. Fang, B.S. He, H.~Liu, and X.M. Yuan.
\newblock Generalized alternating direction method of multipliers: new
  theoretical insights and applications.
\newblock {\em Math. Program. Comput.}, 7(2):149--187, 2015.

\bibitem{He2015NM}
B.S. He and X.M. Yuan.
\newblock On non-ergodic convergence rate of douglas-rachford alternating
  direction method of multipliers.
\newblock {\em Numer. Math.}, 130:567--577, 2015.

\bibitem{Hebs2015SIAMJO}
B.S. He, L.S. Hou, and X.M. Yuan.
\newblock On full jacobian decomposition of the augmented lagrangian method for
  separable convex programming.
\newblock {\em SIAM J. Optim.}, 25:2274--2312, 2015.

\bibitem{Hebs2015IMAJNA}
B.S. He, M.~Tao, and X.M. Yuan.
\newblock A splitting method for separable convex progamming.
\newblock {\em IMA J. Numer. Anal.}, 31:394--426, 2015.

\bibitem{Chench2016MP}
C.H. Chen, B.S. He, Y.Y. Ye, and X.M. Yuan.
\newblock The direct extension of admm for multi-block convex minimization
  problems is not necessarily convergent.
\newblock {\em Math. Program.}, 155:57--79, 2016.

\bibitem{Caixj2013}
X.J. Cai, G.Y. Gu, B.S. He, and X.M. Yuan.
\newblock A proximal point algorithm revisit on the alternating direction
  method of multipliers.
\newblock {\em Sci. China Math.}, 56:2179--2186, 2013.

\bibitem{Alvarez2001}
F.~Alvarez and H.~Attouch.
\newblock An inertial proximal method for maximal monotone operators via
  discretization of a nonlinear oscillator with damping.
\newblock {\em Set-Valued Anal.}, 9:3--11, 2001.

\bibitem{attouch:hal-01708905}
H.~Attouch and A.~Cabot.
\newblock Convergence of a relaxed inertial proximal algorithm for maximally
  monotone operators.
\newblock hal-01708905, February 2018.

\bibitem{Attouch2018SJO}
H.~Attouch and A.~Cabot.
\newblock Convergence rates of inertial forward-backward algorithms.
\newblock {\em SIAM J. Optim.}, 28(1):849--874, 2018.

\bibitem{Attouch2019AMO}
H.~Attouch and A.~Cabot.
\newblock Convergence of a relaxed inertial forward-backward algorithm for
  structured monotone inclusions.
\newblock {\em Appl. Math. Optim.}, pages 1--52, 2019.

\bibitem{Bot2016NA}
R.I. Bo\c{t} and E.R. Csetnek.
\newblock An inertial forward-backward-forward primal-dual splitting algorithm
  for solving monotone inclusion problems.
\newblock {\em Numer. Algorithms}, 71:519--540, 2016.

\bibitem{Cui2019}
F.Y. Cui, Y.C. Tang, and Y.~Yang.
\newblock An inertial three-operator splitting algorithm with applications to
  image inpainting.
\newblock {\em Appl. Set-Valued Anal. Optim.}, 1(2):113--134, 2019.

\bibitem{Chen2015}
C.H. Chen, R.H. Chan, S.Q. Ma, and J.F. Yang.
\newblock Inertial proximal admm for linearly constrained separable convex
  optimization.
\newblock {\em SIAM J. Imaging Sci.}, 8(4):2239--2267, 2015.

\bibitem{Attouch2008PJO}
H.~Attouch and M.~Soueycatt.
\newblock Augmented lagrangian and proximal alternating direction methods of
  multipliers in hilbert spaces, applications to games, pde's and control.
\newblock {\em Pacific. J. Optim.}, 5(1):17--37, 2008.

\bibitem{Bot2016MTA}
R.I. Bo\c{t} and E.R. Csetnek.
\newblock An inertial alternating direction method of multipliers.
\newblock {\em Minimax Theory Appl.}, 1:29--49, 2016.

\bibitem{Bot2015AMC}
R.I. Bo\c{t}, E.R. Csetnek, and C.~Hendrich.
\newblock Inertial douglas-rachford splitting for monotone inclusion problems.
\newblock {\em Appl. Math. Comput.}, 256:472--487, 2015.

\bibitem{bauschkebook2017}
H.H. Bauschke and P.L. Combettes.
\newblock {\em Convex Analysis and Monotone Operator Theory in Hilbert Spaces}.
\newblock Springer, London, second edition, 2017.

\bibitem{Mainge2008JCAM}
P.E. Maing\'{e}.
\newblock Convergence theorems for inertial km-type algorithms.
\newblock {\em J. Comput. Appl. Math.}, 219:223--236, 2008.

\bibitem{Yuanx2013PJO}
X.~Yuan and J.~Yang.
\newblock Sparse and low-rank matrix decomposition via alternating direction
  methods.
\newblock {\em Pacific J. Optim.}, 9(1):167--180, 2013.

\bibitem{Candes2009ACM}
E.J. Cand\`{e}s, X.~Li, Y.~Ma, and J.~Wright.
\newblock Robust principal component analysis?
\newblock {\em J. ACM}, 58(1):1--37, 2009.

\bibitem{Bouwmans2014}
T.~Bouwmans and E.H. Zahzah.
\newblock Robust pca via principal component pursuit: A review for a
  comparative evaluation in video surveillance.
\newblock {\em Comput. Vis. Image Under.}, 122:22--34, 2014.

\bibitem{Masq2018}
S.Q. Ma and N.S. Aybat.
\newblock Efficient optimization algoirhms for robust principal component
  analysis and its variants.
\newblock {\em Procedings of the IEEE}, 106(8):1411--1426, 2018.

\bibitem{Chandrasekaran2011SIAM}
V.~Chandrasekaran, S.~Sanghavi, P.A. Parrilo, and A.S. Willsky.
\newblock Rank-sparsity incoherence for matrix decompositions.
\newblock {\em SIAM J. Optim.}, 21(2):572--596, 2011.

\end{thebibliography}

\end{document}